\documentclass{article}
\usepackage[english]{babel}
\usepackage{amsmath}
\usepackage{amssymb}
\usepackage{amsthm}
\newtheorem{prop}{Proposition}
\newtheorem{theorem}{Theorem}

\newtheorem{definition}{Definition}[section]
\newtheorem{remark}{Remark}

\usepackage{tikz-cd}
\usepackage{xcolor}
\usepackage[sorting=none, backend=biber]{biblatex}

\usepackage[en-US]{datetime2}

\usepackage[autostyle=true]{csquotes} 
\usepackage{graphicx} 
\usepackage{amsfonts}

\title{Almost Trivial Units in Group Rings}
\date{}
\date{Efrat Shani, Anupam Srivastav}
\begin{document}
\maketitle
\footnote{This paper is partly based on Shani's PhD Thesis \cite{shani2024trivial} }
Abstract: 

We find all ring of integers, $R$, and finite groups $G$, such that $RG$ has almost trivial units. Moreover, for such a group ring we determine the Swan subgroup of its locally free classgroup. 

\section{Introduction}
Throughout this, let $|G|=n>1$ be a finite group. Let $R=\mathcal{O}_K$ be the ring of integers of $K$, a number field. Let $(\Sigma)$ be an ideal in the group ring $RG$ generated by $\sum_{g\in G} g$. Denote $\eta_0=\frac{\sum_{g\in G}g}{n}$, $\eta'=1-\eta_0$. Those are orthogonal idemponets in the group algebra $KG$. That lets us write $KG=KG\eta_0\oplus KG\eta'$. Because $g\eta_0=\eta_0$ for all $g\in G$, $KG\eta_0=K\eta_0$, we have $KG=K\eta_0\oplus KG\eta'$. As a subalgebra, we have $R\eta_0+R\eta'\leq K\eta_0+K\eta'$. Let $u,v\in U(R)$ so $u\eta_0+v\eta'$ is a unit in $R\eta_0+R\eta'$. ($U(R\eta_0+R\eta')=U(R)\eta_0+U(R)\eta'=\{u\eta_0+v\eta' | u,v\in U(R)\}$). 
\begin{prop}
Let $u,v\in U(R)$. If $u\equiv v$ mod $n$, then $u\eta_0+v\eta'$ is a unit in $RG$. 
\end{prop}
\begin{proof}
    $u\eta_0+v\eta'=u\eta_0+v(1-\eta_0)=v+(u-v)\eta_0=v+\frac{u-v}{n}\Sigma$. It is in $RG$, as $u\equiv v$ mod $n$. Moreover, $(u^{-1}\eta_0+v^{-1}\eta')(u\eta_0+v\eta')=\eta_0+\eta_1=1$, by property of orthogonal idemponets.
\end{proof} 
 Define $U(RG)$ as the units of $RG$. $TU(RG)=\{ug | u\in U(R), g\in G\}$. as the Trivial units of $RG$. $ATU(RG)=\{(u\eta_0+v\eta')g | u\equiv v$ mod $nR$, $u,v\in U(R)$\} as Almost Trivial Units of $RG$. \\
Also note, $TU(RG)\leq ATU(RG)\leq U(RG)$ as groups. 
\\
We say $RG$ has only trivial units if $U(RG)=TU(RG)$ and $RG$ has only almost trivial units is if $U(RG)=ATU(RG)$. 
\begin{prop}
     $\alpha\in U(RG)$ is in $ATU(RG)$ if and only if $\alpha=vg+k\Sigma$ for some $v\in U(R), g\in G, k\in R$ and $\Sigma=\sum_gg$. 
\end{prop}
\begin{proof}
Assume $\alpha\in U(RG)$ is in $ATU(RG)$ then $\alpha=(u\eta_0+v\eta')g=u\eta_0+vg\eta'=u\eta_0+vg(1-\eta_0)=vg+(u-vg)\eta_0=vg+(u-v)\eta_0$. Since $u=v$ mod $n$, we can let $k=\frac{u-v}{n}$ in $R$. So we have $\alpha=vg+k\Sigma$. \\ Now, if $\alpha=vg+k\Sigma=vg+nk\eta_0=vg(1-\eta_0)+nk\eta_0+v\eta_0=vg\eta'+(nk+v)\eta_0$. Take, $u=nk+v$. Which means $u=v$ mod $n$. Now since $\alpha$ is a unit there is an $\alpha^{-1}$ such that $(vg\eta'+(nk+v)\eta_0)\alpha^{-1}=1$, for some $x\in R$, we let $\alpha^{-1}=(vg)^{-1}\eta'+x\eta_0$. This results in $\eta'+ux\eta_0=1$. As $\eta'=1-\eta_0$ is forces $ux=1$ so $u$ must be a unit.  
\end{proof} 

    That leads us to the following definition: 
\begin{definition} [Reduced (Truncated) Group Rings]
    We defined the reduced (truncated) group ring as the quotient group ring, $RG_t:=RG\big/_{(\Sigma)}$. Where $(\Sigma)=(\Sigma)RG$ is a principal ideal of $RG$ 
\end{definition}

In the $K$-algebra $KG_t=\frac{K\eta_0\oplus KG\eta'}{\Sigma}\cong KG\eta'$. \\The trivial units of $RG_t$ are defined as $TU(RG_t)=\{ug+(\Sigma) | u\in U(R), g\in G\}$.

We can now conclude, if $RG_t$ has only trivial units so $U(RG_t)=TU(RG_t)$, $RG$ has only almost trivial units, $U(RG)=ATU(RG)$.  
\\
The motivation for defining Reduced (Truncated) Group Rings, comes from the following Milnor's Square, with classically  $R=\mathbb Z$:

\[
  \begin{tikzcd}
    RG \arrow{r}{} \arrow{d}{\epsilon} & RG_t \arrow{d}{\bar{\epsilon}} \\
    R \arrow{r}{\pi}  & \bar{R}  \\
  \end{tikzcd}
\]
Where $\bar{R}$ is $\frac{R}{nR}$, $n=|G|$. $\bar{\epsilon}$ is induced by the augmention map $\epsilon: \sum r_gg\rightarrow \sum r_g$. $\bar{\epsilon},\pi$ are surjective. \\
There is an exact Mayer-Vietoris sequence. $$U(RG)\rightarrow U(RG_t) \times U(R) \xrightarrow{h}  U(\bar{R}) \xrightarrow{\delta} K_0(RG)\rightarrow K_0(RG_t)\oplus K_0(R)\rightarrow 0.$$ Where $K_0(RG)$ is the Grothendieck group. It it known by, \cite{greither1999swan}, \cite{taylor1984classgroups} that $$h(U(RG_t) \times U(R))=\bar{\epsilon}(U(RG_t))\times \pi(U(R)).$$ $\delta(\bar{r})$ is the class of the Swan module $\langle r,\sum\rangle$, in $K_0(RG)$. 

\begin{definition} [Swan Subgroups] 
  $\delta(U(\bar{R}))=T(RG)$, the Swan subgroup. Using the properties of exact sequence, $T(RG)=\frac{U(\bar{R})}{h(U(RG_t) \times U(R))}$.
\end{definition}
From \cite{taylor1984classgroups} Taylor's theorem 2.5, we see that if $G$ is an elementary abelian $p$-group $C_{p^n}$ where $p$ is a prime number then the order of the Swan subgroup is $p^{n-1}$ for $p\neq 2$. And $2^{n-2}$ when $p=2$. \\ 
If $RG_t$ has only trivial units, the Swan subgroup is $T(RG)=\frac{U(\bar{R})}{\pi(U(R))}$. \\ For example, it was calculated first in Rich's thesis \cite{rich2020units} that the Swan subgroup of $\mathbb Z (C_3^n\times C_2^m)$ for $n,m\geq1$ is a group of order $(3^{n-1})(2^{m-1})$. 

\textbf{Dirichlet's Unit Theorem}
$U(R)$ is a finitely generated abelian group of a finite $\mathbb Z$-rank. We denote the rank as, $rk(U(R))$ and it is $r_1+r_2-1$  where  $r_1$ is the number of real embedding from $K$ to $\mathbb Q$, and $r_2$ is the number of imaginary embedding. It is known $[K:\mathbb Q]=r_1+2r_2$. 
\\

    The rank of trivial units of $RG$ ($RG_t$), is the same as the rank of $U(R)$ since $G$ is a finite group.  
\\
Therefore, $TU(RG_t), TU(RG)$ are also finitely generated abelian group of a $\mathbb Z$-rank, $rk(TU(RG_t))=rk(TU(RG))=r$. While, $rk(ATU(RG))=2r$. 

\begin{theorem} [Higman]
$\mathbb Z G$ has only trivial units if and only if,\\  $G\cong C_3^m\times C_2^n, C_4^m\times C_2^n, Q_8\times C_2^n$ where $n,m\geq 0$.  \cite{higman1940units}
\end{theorem} 

\begin{theorem} [Rich]
   If $G$ is an abelian group from Higman's Theorem above then,  $\mathbb ZG_t$ has only trivial units. \cite{rich2020units}
\end{theorem}
\begin{theorem} [Herman and Li]
For $R\neq\mathbb Z$, $RG$ has only trivial units if and only if,\\  $$R=\mathbb Z[i], G\cong C_4^m\times C_2^n;$$  $$R=\mathbb Z[\omega], G\cong C_3^m\times C_2^n;$$  $$R=\mathcal{O}_{\mathbb Q(\sqrt{-d})}, G\cong C_2^n;$$ for $i=\sqrt{-1}, \omega=\frac{-1+i\sqrt{3}}{2}$ ,$n,m\geq 0$, $d$ square free $d>1, d\neq 3$. \cite{herman2006trivial}
\end{theorem} 

\begin{theorem} 
     $RG_t$ has only trivial units if and only if one of the following holds:\\
    (1) $G=C_2$ \\
    (2) $G=C_3$, $R=\mathcal{O}_K$ where $K$ is a totally real field . \\
    (3)\footnote{This was proved in Shani's Thesis} $R$ and $G$ are as in Higman's or Herman and Li's Theorems above.  
   
\end{theorem}

\section{The rank of $U(RG_t)$}
Let $M$ denote a maximal order of $KG$, and $M_t$ a maximal order of $KG_t$. Recall $R=\mathcal O_K$. The $R$-rank of $M$ can be generalized and defined as $rk_R(M)=dim_K(K\otimes_R M)$. (From \cite{reiner2003maximal} pp. 44). And therefore, the rank of units in an order is the same as the rank of the units of a maximal order.\\  Specifically, $rk(U(RG))=rk(U(M))$ and $rk(U(RG_t))=rk(U(M_{t}))$.
\begin{prop}
    If $U(\mathbb ZG_t)=TU(\mathbb ZG_t)$ Then $G$ is $G=
\begin{cases}
    C_2^n \times C_3^m \\ 
    C_2^n \times C_4^m \\
    C_2^n \times Q_8 
\end{cases}$
\\For  $n,m\geq 0$.   
\end{prop}
\begin{proof}
    
Let $K=\mathbb Q$. Recall, $$\mathbb QG\cong \mathbb Q \oplus \mathbb QG_t$$ so $M\cong \mathbb Z \oplus M_t$. Since $rk(U(\mathbb Z))=0$, if we assume $U(\mathbb ZG_t)=TU(\mathbb ZG_t)$ then it must be the case that $rk(U(\mathbb ZG_t))=0$, and also $rk(U(\mathbb ZG))=0$. From Higman we know $rk(U(\mathbb ZG))=0$ if and only if $G=C_2^n \times C_3^m , C_2^n \times C_4^m , C_2^n \times Q_8 .$ Where $n,m\geq0$. We will call those Higman's Groups. 
\end{proof} 
Therefore, if $\mathbb ZG_t$ has only trivial units $G$ must be one of the Higman's Groups. 

\begin{remark}
Since $\mathbb Z \hookrightarrow R $, if $RG_t$ has only trivial units then $G$  
is one of Higman's Groups.\\

\end{remark}

\begin{prop}
    If $RC_4$ has non trivial units, then $RQ_{8_t}$ has non trivial units.
\end{prop}
\begin{proof}

Let $C_4$ be a cyclic subgroup of $Q_8$ of order $4$.
Note the embedding of $RC_4\hookrightarrow RQ_8$ and the surjection $RQ_8\twoheadrightarrow RQ_{8_t}$. So we have the map $RC_4\rightarrow RQ_{8_t}$ via $\sum \alpha_g g \mapsto \sum\alpha_g g +(\Sigma)$. Where $(\Sigma)$ is the ideal generated by $\sum_{g\in Q_8} g$, no non-zero elements of $C_4$ are contained in $(\Sigma)$. Take a non trivial unit in $v\in RC_4$. Its image is $v'$. Assume to the contrary, only trivial units are in $RQ_{8_t}$. So this is a trivial unit, then it is $ug+n\Sigma$ for some $u\in U(R), g\in Q_8, n\in R$. If $g\in C_4$ then the coefficient of $g\in v'$ is $u+n$ while the rest of the element will must have coefficient $n$, so pulled back to $RQ_8$ will result in $ug$, since $RC_4\hookrightarrow RQ_8$ then it imply the unit is trivial in $RC_4$. Therefore, $g\in Q_8$ not in $C_4$. Then, pulled back to $RQ_8$ it is $ug$ with $u\in U(R), g\in Q_8$ but $g$ is not in the cyclic subgroup $C_4$. So that trivial unit does not have a pre-image in $RC_4$. \\ 
We now examine the case $rk(U(R))=0$ and $R\neq \mathbb Z$. We know this means $R$ is the maximal order of a quadratic imaginary number field. 
\end{proof}
 \begin{prop}
     Let $R:= \mathcal O_{\mathbb Q(\sqrt{-d})}$ where $d>0$, square free integer. If $G$ is an abelian Higman's group such that $RG_t$ has only trivial units, then $RG$ are only the group rings that are  mentioned in Theorem 3. 
 \end{prop}
\begin{proof}
    
From Herman and Li we can deduce that $rk(U(RG))=0$ if and only if $RG$ has only trivial units. Since $RG_t$ has only trivial units then $rk(U(RG_t))=rk(U(R))=0$, so also the rank of $RG$ must be zero (The decomposition $RG\cong R\oplus M_t$ grantees that).  So now that the $rk(U(RG))=0$, we use Herman \& Li theorem to conclude the group rings that we can have. To further demonstrate that results, consider what happen when $G=C_4, C_3$:
\\ If $G\cong C_4$, the decomposition would be $\mathbb Q(\sqrt{-d})C_4\cong \mathbb Q(\sqrt{-d})\eta_0 \oplus \mathbb Q(\sqrt{-d})C_4\eta' $. $M\cong \mathcal{O}_{\mathbb Q(\sqrt{-d})}\eta_0 \oplus M_t$. Under the assumption $RG_t$ has only trivial units and $rk(U(R))=0$, it means the $rk( U(M_t))=0$. That can only happen if $d=1$ so $K=\mathbb Q(i)$. 
 \\ Similarly, if $G\cong C_3$, the decomposition would be $\mathbb Q(\sqrt{-d})C_3\cong \mathbb Q(\sqrt{-d})\eta_0 \oplus \mathbb Q(\sqrt{-d})C_3\eta'$. $M\cong \mathcal{O}_{\mathbb Q(\sqrt{-d})}\eta_0 \oplus M_t$. Under the assumption $RG_t$ has only trivial units and $rk(U(R))=0$, it means the $rk( U(M_t))=0$. That only happens when $d=3$ so $K=\mathbb Q(\omega)$. 
\end{proof}

\begin{prop}
        If $rk(U(R))\geq 1$ and $RG_t$ has only trivial units then either $G=C_2$ or $G=C_3$ and $K$ must be a totally real field.   
\end{prop}
\begin{proof}
    
First, consider only abelian Higman's Groups. The assumption that $RG_t$ has only trivial units means if we denote $rk(U(RG_t))=r$ then $rk(U(RG))=2r$. $$KG\cong K\eta_0\oplus KG\eta'\cong K\eta_0 \oplus KG_t$$ Moreover, we can decompose $KG\eta'$ into $\oplus_{l=1}^{l=i}KG\eta_i$ where $\eta_i$ are orthogonal, primitive idemponets of $KG$ and $\eta_i\neq \eta_0$. Then, from Higman we know $\oplus_{l=1}^{l=i}KG\eta_i \cong \oplus_{l=1}^{l=i}K(\zeta_i)$ via $g\eta_i\mapsto \zeta_i$. Where $\zeta_i$ is a root of unity. Since each $K(\zeta_i)$ is an extension field of $K$, and the rank of $U(R)$ is $r$. The rank of the maximal order of each extension must be at least $r$. Therefore, $i=1$. For conveniences, $\zeta_1=\zeta$. We can only have the decomposition $KG\cong K\eta_0 \oplus KG\eta_1 \cong K \oplus K(\zeta)$ with $\eta_0+\eta_1=1$. Moreover, for the dimension over $K$ to match, we must have $[K(\zeta):K]=|G|-1$.

If $|G|=n>3$. Then, $[K(\zeta):\mathbb Q]=[K(\zeta):K][K:\mathbb Q]=(n-1)(r_1+2r_2)=R_1+2R_2$. Where $r_1$ is the real embedding of $K$ into $\mathbb Q$, $r_2$ is the complex embedding of $K$ into $\mathbb Q$, $R_1$ is the real embedding of $K(\omega)$ into $\mathbb Q$, $R_2$ is the complex embedding of $K(\zeta)$ into $\mathbb Q$. Moreover, since the rank of $U(R)$ is the same as the rank of units of $\mathcal{O}_{K(\zeta)}$.  $r_1+r_2-1=R_1+R_2-1$.  Solving this system of equations with the conditions that $r_1,r_2, R_1, R_2$ are non-negative and $n>3$ would lead to a contradiction because $r_1+r_2-1=R_1+R_2-1$ implies $2r_1+2r_2=2R_1+2R_2$. Then subtract this equation from $(n-1)(r_1+2r_2)=R_1+2R_2$ to get $(n-3)r_1+2(n-2)r_2=-R_1$. Since $R_1$ is non-negative, $-R_1$ must be negative or $R_1=0$. However, $n-3>0$ as well as $2(n-2)>0$. So $(n-3)r_1+2(n-2)r_2$ must be positive. (Note that it can't be zero, because we can't have both real and complex embedding zero). 
\\
Moreover, when $G=C_3$, $\zeta=\omega$, $[K(\omega):\mathbb Q]=[K(\omega):K][K:\mathbb Q]=2(r_1+2r_2)=R_1+2R_2$. And $r_1+r_2-1=R_1+R_2-1$. Solving this case in a similar way leads to , $2r_2=-R_1$. Since both must be non-negative, both must be zero for this equality to hold. If $r_2=0$, it means $K$ is be a totally real field. \\
Now, for $G=Q_8$. Simply apply Proposition 4. We know $RC_4$ can only have trivial units when $K=\mathbb Q, \mathbb Q(i)$ from Herman and Li \cite{herman2006trivial}. They provide a non trivial unit in $\mathbb Z[i]Q_{8}$ , and that unit is also seen to be non trivial in $\mathbb Z[i]Q_{8_t}$ as there is no way write it $ug+n\Sigma$ for $u\in U(R), g\in Q_8, n\in R$. The unit is $u=1+(1-x^2)(-1+(1+i)x+(1-i)y)$. ($Q_8=\langle x, y| x^4=1, x^2=y^2, xy=y^{-1}x\rangle $)\\
\end{proof} 

In this section we covered all the potentials $R's$ and $G's$ we can have in order for $RG_t$ to have only trivial units. In the next section we prove each case does have only trivial units.
\section{Trivial Units of $RG_t$}
$RC_{2_t}$ has only trivial units. That is trivial, as $RC_{2_t}\cong R$. Note that it means that $U(RC_2)=ATU(RC_2)$ so,
$$U(RC_2)=\{\frac{u+ug}{2}+\frac{v-vg}{2} | u,v \in U(R), u= v \mod 2, g\in C_2 , g^2=1_G\}$$ 
This section is based on Shani's PhD thesis \cite{shani2024trivial}:
\begin{prop}
    $R(C_2\times C_{2})_t$ has only trivial units when $U(R)$ is finite.
\end{prop}
\begin{proof}
Let $C_{2}\times C_2=\langle f,g | f^2=g^2=1\rangle$. The idempotents are $$\eta_0=\frac{1+f+g+fg}{4}, \eta_1=\frac{1-f+g-fg}{4}, \eta_2=\frac{1+f-g-fg}{4}, \eta_3=\frac{1-f-g+fg}{4}.$$ Therefore, we decompose $KG$ as $$KG\cong K\eta_0\oplus K\eta_1\oplus K\eta_2\oplus K\eta_3,$$ for any number field $K$, as all idempotents are realized over $\mathbb Q$. Then its maximal order would be $$M_2\cong \mathcal{O}_K\eta_0\oplus \mathcal{O}_K\eta_1\oplus \mathcal{O} _K\eta_2\oplus \mathcal{O}_K\eta_3.$$ In the reduced (truncated) case $\eta_0$ would be congruent to zero, so $$M_{2_t}\cong \mathcal{O}_K\eta_1\oplus \mathcal{O} _K\eta_2\oplus \mathcal{O}_K\eta_3.$$ In the case $K=\mathbb Q, \mathbb Q(\sqrt{-d})$ ($d>0 \neq 1,3$ square free), the units in $\mathcal{O}_K$ are only $\pm1$. The amount of units in $M_{2_t}$ can only be $8$, which is the amount of trivial units by definition. All units are trivial in that group ring. 
\\ For $d=1,3$, we chase back a unit to $KG$. Let $u_i\in U(\mathcal{O}_K$). So the image $u_1\eta_1+u_2\eta_2+u_3\eta_3$ first pulled back to $M_2$ we note that any $r\in\mathcal{O}_K$ could have been multiplied with $\eta_0$. We also note we can factor $u_3$, so $$u_3(r'\eta_0+u_1'\eta_1+u_2'\eta_2+1\eta_3)$$ as $u_i'$ are units in $\mathcal{O}_K$ and $r'\in\mathcal{O}_K$. Now, pulled back to $M_2$ $u_3$ will be mapped to a trivial unit by definition, so it is just left to see what $r'\eta_0+u_1'\eta_1+u_2'\eta_2+1\eta_3$ will be pulled back to. That is, $$\frac{r'+u_1'+u_2'+1}{4}1_G+\frac{r'-u_2'+u_3'-1}{4}f+\frac{r'+u_2'-u_3'-1}{4}g+\frac{r'-u_2'-u_3'+1}{4}fg.$$  To be in $\mathcal{O}_KG$ each numerator must be divisible by $4$. Solving this system of equation in$\mod 4\mathbb Z[i],\mod 4\mathbb Z[\omega]$ leads to only 16, 24 possible units respectively which is the amount of trivial units.

\end{proof}
\begin{prop}
    $\mathbb Z[\omega]C_{3_t}$ has only trivial units. 
\end{prop}
\begin{proof}
Denote $C_3=\langle f\rangle$. We first write the idempotents : $$\eta_0=\frac{1+f+f^2}{3}, \eta_1=\frac{1+\omega f+\omega^2f^2}{3}, \eta_2=\frac{1+\omega^2f+\omega f^2}{3}.$$
    Now, write the group algebra decomposition. $$\mathbb{Q}(\omega)C_3\cong \mathbb Q(\omega)\eta_0 \oplus \mathbb Q(\omega)\eta_1 \oplus \mathbb Q(\omega)\eta_2.$$ 
    the reduced (truncated) maximal order is, $$M_{3_t}=\mathbb{Z}[\omega]\eta_1\oplus \mathbb{Z}[\omega]\eta_2.$$ 
    There are $18$ trivial units in $\mathbb Z[\omega]C_{3_t}$ by definition. In the reduced (truncated) maximal order there will be $36$ by counting. Since the units form a subgroup of the unit group of the maximal order, either all $36$ units in the maximal order are in $\mathbb Z[\omega]C_{3_t}$ or only $18$. We provide one example of a unit in the maximal order that is not in $\mathbb Z[\omega]C_{3_t}$. $-\eta_1+\eta_2$ pulled back is $\frac{-1-\omega f-\omega^2f^2}{3}+\frac{1+\omega^2+\omega f^2}{3}=\frac{(-\omega+\omega^2)f+(\omega-\omega^2)f^2}{3}.$ But $\omega-\omega^2$ is not divisible by $3$ in $\mathbb Z[\omega]$. That unit is not in $\mathbb Z[\omega]C_{3_t}$.   Therefore, $\mathbb Z[\omega]C_{3_t}$ has only trivial units.    
\end{proof}    
\begin{prop}
     $\mathbb Z[i]C_{4_t}$ has only trivial units. 
\end{prop}
\begin{proof}
Similarly, denote $C_4=\langle f\rangle$. We first write the idempotents : $$\eta_0=\frac{1+f+f^2+f^4}{4}, \eta_1=\frac{1-f+f^2-f^3}{4}, \eta_2=\frac{1+if-f^2-if^3}{4}, \eta_3=\frac{1-if-f^2+if^3}{4}.$$
Now, write the group algebra decomposition. $$\mathbb{Q}(i)C_4\cong \mathbb Q(i)\eta_0 \oplus \mathbb Q(i)\eta_1 \oplus \mathbb Q(i)\eta_2\oplus \mathbb Q(i)\eta_3.$$ Then, the reduced maximal order, $$M_{4_t}\cong \mathbb Z[i]\eta_1 \oplus\mathbb Z[i]\eta_2\oplus \mathbb Z[i]\eta_3.$$ The units can be described as $u_1\eta_1+u_2\eta_2+u_3\eta_4=u_3(u_1'\eta_1+u_2'\eta_2+\eta_3)$ with $u_i\in U(\mathbb Z[i])$, and we need to show that the pre-image of $u_1'\eta_1+u_2'\eta_2+\eta_3$ is in $\mathbb Z[i]C_{4_t}$ only $4$ times. To reach an overall $16$ units, total amount of trivial. When pulled back, we get $$\frac{r+u_1+u_2+1}{4}+\frac{r-u_1+iu_3-i}{4}+\frac{r+u_1-u_2-1}{4}+\frac{r-u_1-iu_2+i}{4}.$$ Where $r\in\mathbb Z[i]$. This is matter of solving:
\begin{align}
r+u_1+u_2+1&=0 \mod 4\mathbb Z[i] \\
r-u_1+iu_2-i&= 0  \mod 4\mathbb Z[i]  \\
r+u_1-iu_2-1&= 0  \mod 4\mathbb Z[i]  \\
r-u_1-iu_2+i&= 0  \mod 4\mathbb Z[i] 
\end{align}

Subtracting (2) and (4), $2iu_2-2i=0 \mod 4\mathbb Z[i]$. Multiplying by $(2i)^{-1}$ we are left with $u_2-1=0 \mod 2\mathbb Z[i]$, $u_2=1,-1$. That leaves us with $32$ options. We can narrow it down further. Consider subtracting equations (1) and (2). Then $2u_1+(1-i)u_2+1+i=0 \mod 4\mathbb Z[i]$. Also, we can write it as $-i(1+i)^2u_1-i(1+i)u_2+(1+i)=0 \mod (1+i)^4\mathbb Z[i]$. Which we can divide by $(1+i)$ to get,  $-i(1+i)u_1-iu_2+1=0 \mod (1+i)^3\mathbb Z[i]$. 

Case 1: $u_2=1$ then ,$-i(1+i)u_1-i+1=0 \mod (1+i)^3\mathbb Z[i]$ but then we can divide by another $(1+i)$ to get $-iu_1-i=0 \mod (1+i)^2\mathbb Z[i]$, so $u_1=1 \mod 2\mathbb Z[i]$ when $u_2=1$. 

Case 2: $u_2=-1$ then $-i(1+i)u_1+i+1=0 \mod (1+i)^3$. Again, divide by $(1+i)$ so now $u_1=i \mod 2\mathbb Z[i]$ when $u_2=-1$. Now we have $4$ choices left. With $u_3$, that leads to $16$ units that can be in $\mathbb Z[i]C_{4_t}$ which counts for the trivial units only.
\end{proof}
\begin{prop}
        $\mathbb ZQ_{8_t}$ has only trivial units. 
    \end{prop}
\begin{proof}
$$Q_8=\langle x,y | x^4=e,  x^2=y^2, z=xy=y^{-1}x\rangle.$$ The idempotents are $$\eta_0=\frac{e+\bar{e}+x+\bar{x}+y+\bar{y}+z+\bar{z}}{8}, \eta_1=\frac{e+\bar{e}+x+\bar{x}-y-\bar{y}-z-\bar{z}}{8}, $$$$\eta_2=\frac{e+\bar{e}-x-\bar{x}+y+\bar{y}-z-\bar{z}}{8}, \eta_3=\frac{e+\bar{e}-x-\bar{x}-y-\bar{y}+z+\bar{z}}{8}, \eta_4=\frac{4e-4\bar{e}}{8}.$$
The group algebra decomposition is: $$\mathbb Q Q_{8}=\mathbb Q\eta_0 \oplus \mathbb Q\eta_1\oplus \mathbb Q\eta_2 \oplus \mathbb Q\eta_3 \oplus \mathbb QG\eta_4.$$ Denote the Hamiltonian division algebra $$\mathbb H_{\mathbb Q}=\{a1+bi+cj+dk |a,b,c,d\in\mathbb Q,   i^2=j^2=k^2=-1\}.$$ Note the isomorphism, $\mathbb QG\eta_4 \cong \mathbb H_{\mathbb Q}$ via $\eta_4\mapsto 1, x\mapsto i, y\mapsto j, z\mapsto k$. Then, a maximal order $M_8\cong \mathbb Z^{\oplus4} \oplus \mathbb H_{\mathbb Z}$, and a reduced maximal order will be $$M_{8_t}\cong \mathbb Z^{\oplus3} \oplus \mathbb H_{\mathbb Z}.$$ Counting the units in this structure we have $64$ units. We show that under the map only $16$ are in $\mathbb ZQ_{8_t}$. Consider, $r\in\mathbb Z, u_1,u_2,u_3\in U(\mathbb Z), u_4\in U(\mathbb H_{\mathbb Z})$. Since $u_4$ will be pulled back to elements in $Q_8$, we get a different expression depending on $u_4$ when pulled back to $\mathbb QG$. We consider each case. For $u_4=e, \bar{e}$. The expression is pulled back to: 
\begin{equation*}
\begin{split}
    \frac{1}{8}\big((r+u_1+u_2+u_3+4u_4)e+(r+u_1+u_2+u_3-4u_4)\bar{e}+(r+u_1-u_2-u_3)x+(r+u_1-u_2-u_3)\bar{x}\\+(r-u_1+u_2-u_3)y+(r-u_1+u_2-u_3)\bar{y}+(r-u_1-u_2+u_3)z+(r-u_1-u_2+u_3)\bar{z}\big)        
\end{split}
\end{equation*}
    We solve the following system of equation:
    \begin{align*}
r+u_1+u_2+u_3+4u_4&=0 \mod 8 \\
r+u_1+u_2+u_3-4u_4&= 0  \mod 8  \\
r+u_1-u_2-u_3&= 0  \mod 8  \\
r-u_1+u_2-u_3&= 0  \mod 8 \\
r-u_1-u_2+u_3&= 0  \mod 8
\end{align*}

Adding all equations, we are left with $5r+u_1+u_2+u_3=0 \mod 8$ .We can strict $r\in [0,7]$ as we are working in module $8$. Knowing $u_i=\pm 1$ the values of the equations is in the range $[-3,38]$. Since we need it to equal $0 \mod 8$ we are now only left with five possible answers $\{0,8,16,24,32\}$. $u_1+u_2+u_3$ can only add up to $-3,-1,1,3$ (because they are all $\pm 1$) so we can never get $0$. For The answers to be $8,32$ we choose $r=u_1=u_2,u_3=1$ or  $r=7 , u_1=u_2=u_3=-1$. For the answer to be $16$, $r=3$, $u_1+u_2+u_3=1$ there's $3$ such possibilities, but (without loss of generality) choose $u_1=u_2=1, u_3=-1$ then some equations are not satisfied.
Similarly, for the answer to be $24$, $r=5, u_1+u_2+u_3=-1$ and again, if we choose $u_1=u_2=-1, u_3=1$ then some equations are not satisfied.

Overall when $u_4=e, \bar{e}$ there are four possible units in $\mathbb ZQ_8$.

Similar work is done when solving the other cases where $u_4$ is the other elements of $Q_8$.

We add all possible units and reach $16$ which already counts for the trivial units. So there are only trivial units in $\mathbb ZQ_{8_t}$. 
\end{proof}
    
    Consider the isomorphism, $\theta$ $$K(G\times C_n) \cong KG\otimes_K KC_n= KG\otimes_K K\eta_0 \oplus KG\otimes_K K\eta_1\oplus ... \oplus KG\otimes_K K\eta_{l}$$ and the projection, $\pi$ of  $$KG\otimes_K K\eta_0 \oplus KG\otimes_K K\eta_1\oplus ... \oplus KG\otimes_K K\eta_{l}$$ to $$KG_t\otimes_K K\eta_0 \oplus KG\otimes_K K\eta_1\oplus ... \oplus KG\otimes_K K\eta_{l}.$$ 
    \[
\begin{tikzcd}[column sep=small]
K(G\times C_n) \arrow{r}{\theta}  \arrow{rd}{\pi\theta} 
  & KG\otimes KC_n \arrow{d}{\pi} \\
    & KG_t\otimes K\eta_0 \oplus KG \otimes K\eta_1 \oplus ... \oplus KG \otimes K\eta_{l}
\end{tikzcd}
\]Where $\eta_i$ are orthogonal idempotents of $KC_n$, with $\eta_0=\frac{\sum_{f\in C_n}f}{n}$. The map, $\pi\theta$ from $K(G\times C_n)$ to $KG_t\otimes_K K\eta_0 \oplus... \oplus KG\otimes_K K\eta_l$ is surjective with $ker(\pi\theta)=\Sigma$. Therefore, we have the isomorphism, $$K(G\times C_n)_t\cong KG_t\otimes_K K\eta_0 \oplus KG\otimes_K K\eta_1\oplus ... \oplus KG\otimes_K K\eta_l.$$ Let $N=\mathcal{O}_K G_t\otimes_{\mathcal{O}_K} \mathcal{O}_K\eta_0 \oplus \mathcal{O}_KG\otimes_{\mathcal{O}_K} \mathcal{O}_K\eta_1\oplus ... \oplus \mathcal{O}_KG\otimes_{\mathcal{O}_K} \mathcal{O}_K\eta_l$ be an order in $KG_t\otimes_K K\eta_0 \oplus KG\otimes_K K\eta_1\oplus ... \oplus KG\otimes_K K\eta_l$. \\ Unit in $\mathcal{O}_K(G\times C_n)_t$ are the pre-image of units in $N$ that are in $\mathcal{O}_K(G\times C_n)_t$. That is , $U(\mathcal{O}_K(G\times C_n)_t=\mathcal{O}_K(G\times C_n)_t \cap (\pi\theta)^{-1}(N)$ .
    \begin{prop}
        Let $|G|>2$. If $RG$ and $RG_t$ has only trivial units then $R(G\times C_2)_t$ has only trivial units. 
    \end{prop}
\begin{proof}
    We use the map above, when $n=2$, denote $C_2=\langle f\rangle$, and any elements of $G$ are $g$. $N=\mathcal{O}_K G_t\otimes_{\mathcal{O}_K} \mathcal{O}_K\eta_0 \oplus \mathcal{O}_KG\otimes_{\mathcal{O}_K} \mathcal{O}_K\eta_1$. 
  
  Let $x\in R(G\times C_{2})_{_t}$. Then $\pi\theta(x)=(\sum_g \alpha_gg+n\sum) \otimes \eta_0 +\sum_g \beta_gg\otimes \eta_1$ where $n,\alpha_g,\beta_g \in R$. If we pull this back to $K(G\times C_{2})_{_t}$ we see $$x=\frac{1}{2}(\sum(\alpha_g+\beta_g g)+n\sum +\sum(\alpha_g-\beta_g gf)).$$ That means elements in $R(G\times C_2)$ must satisfy the condition that $\alpha_g+n=\beta_g$ mod $2R$ for all $g\in G$. So units in $R(G\times C_{2})_{_t}$ must satisfy that condition. 
  
  $U(N)=(ug+m\sum)\otimes \eta_0+vh\otimes \eta_1=vh(u'g'+m\sum\otimes \eta_0+1_G\otimes \eta_1)$. $vh$ is a trivial unit in $R(G\times C_{2})_{_t}$ so we are left to verify when $u'g'+m\sum\otimes \eta_0+1_G\otimes \eta_1$ satisfy the condition above. That is when, $u'g'+m\sum=1_G$ mod $2RG$.

  If,  $m=2k+1$, then $u'g'+m\sum=1_G$ mod $2RG$. If $g'\neq 1_G$, So comparing coefficients of $G$ tells us, $u'+m=0$, $m=0$, $m=1$ mod $2R$. That is a contradiction, when $|G|>2$.
  
  When, $m=2k$. Then this is the same as $u'g'=1_G$ mod $2RG$. That force $g'=1_G$ and $u'=1$ mod $2R$. $u'=1$ mod $2R$ implies $u'=\pm 1$ (For all possible $R$). We now can write $vh(\pm 1\otimes \eta_0 +1\otimes \eta_1)$. This is simply $vh$ or $-vhf$. Those are trivial units by definition. 
  \end{proof}
\begin{prop}
            Let $|G|>2$, $R=\mathbb Z[\omega]$. If $RG$ and $RG_t$ has only trivial units then $R(G\times C_3)_t$ has only trivial units. 
\end{prop}
\begin{proof} Similarly, denote $C_3=\langle f\rangle$, and any elements of $G$ are $g$. $$N=\mathbb Z[\omega] G_t \otimes_{\mathbb Z[\omega]}  \mathbb Z[\omega]\eta_0 \oplus \mathbb Z[\omega] G \otimes_{\mathbb Z[\omega]}  \mathbb Z[\omega]\eta_1\oplus \mathbb Z[\omega]G \otimes_{\mathbb Z[\omega]}  \mathbb Z[\omega]\eta_2.$$ $U(N)=(u_1g_1+s\sum)\otimes \eta_0+u_2g_2 \otimes \eta_1+u_3g_3\otimes \eta_2= u_3g_3((u'_1g'_1+s'\sum)\otimes \eta_0+u'_2g'_2 \otimes \eta_1+1\otimes \eta_2)$. 
We check the condition for this element to have a pre-image in $\mathbb Z[\omega](G\times C_{3})_{_t}$.  $u_3g_3$ is a trivial unit, while $$(u'_1g'_1+s'\sum)(\frac{1+f+f^2}{3})+u'_2g'_2(\frac{1+\omega f+\omega^2 f^2}{3})+(\frac{1+\omega^2f+\omega f^2}{3})$$ corresponds to, $$\frac{u'_1g'_1+s'\sum+u'_2g'_2+1}{3}1+\frac{u'_1g'_1+s'\sum+\omega u'_2g'_2+\omega^2}{3}f+\frac{u'_1g'_1+s'\sum+\omega^2 u'_2g'_2+\omega}{3}f^2.$$ 


That means we need to solve the following equations: 
\begin{align*}
u'_1g'_1+s'\sum+u'_2g'_2+1&=0 \mod 3\mathbb Z[\omega] \\
u'_1g'_1+s'\sum+\omega u'_2g'_2+\omega^2&=0  \mod 3\mathbb Z[\omega]  \\
u'_1g'_1+s'\sum+\omega ^2u'_2g'_2+\omega&= 0  \mod 3\mathbb Z[\omega]  
\end{align*}

When solving it, we see there are only $3$ possibilities. $g'_1=g'_2=1_G, u'_1=u'_2=1$ or $g'_1=g'_2=1_G, u'_1=\omega^2, u'_2=\omega$ or  $g'_1=g'_2=1_G, u'_1=\omega, u'_2=\omega^2$. With the trivial unit, $u_3g_3$ we have a total of $18|G|$ choices of units, that is also the count for the amount of trivial units. 
\end{proof} 
\begin{prop}
            Let $|G|>2, R=\mathbb Z[i]$.  If $RG$ and $RG_t$ has only trivial units then $R(G\times C_3)_t$ has only trivial units. 
\end{prop}
\begin{proof} Finally, denote $C_4=\langle f\rangle$, and any elements of $G$ as $g$ 
$$N=\mathbb Z[i] G_t \otimes_{\mathbb Z[i]}  \mathbb Z[i]\eta_0 \oplus \mathbb Z[i] G \otimes_{\mathbb Z[i]}  \mathbb Z[i]\eta_1\oplus \mathbb Z[i]G \otimes_{\mathbb Z[i]}  \mathbb Z[i]\eta_2 \oplus \mathbb Z[i]G \otimes_{\mathbb Z[i]}  \mathbb Z[i]\eta_3.$$ 
$U(N)=(u_1g_1+s\sum)\otimes \eta_0+u_2g_2 \otimes \eta_1+u_3g_3\otimes \eta_3+u_4g_4\otimes \eta_3= u_4g_4((u'_1g'_1+s'\sum)\otimes \eta_0+u'_2g'_2 \otimes \eta_1+u'_3g'_3\otimes \eta_2+1\otimes \eta_3)$.
We again check the conditions for this element to have a pre-image in $\mathbb Z[i](G\times C_{3})_{_t}$.
Which means we need to solve the system of equations: 
\begin{align*}
u'_1g'_1+s'\sum+u'_2g'_2+u'_3g'_3+1&=0 \mod 4\mathbb Z[i] \\
u'_1g'_1+s'\sum-u'_2g'_2+iu'_3g_3-i&=0  \mod 4\mathbb Z[i]  \\
u'_1g'_1+s'\sum+ u'_2g'_2-u'_3g_3-1&= 0  \mod 4\mathbb Z[i]  \\
u'_1g'_1+s'\sum-u'_2g'_2-iu'_3g'_3+i&= 0  \mod 4\mathbb Z[i]
\end{align*}

When solved, we only have $4$ possible options. That with  $u_4g_4$ results in $16|G|$ units that can be in $\mathbb Z[i](G\times C_{4})_{_t}$. This counts already for all trivial units, therefore there can only be trivial units in $\mathbb Z[i](G\times C_{4})_{_t}$. 
\end{proof}
\begin{remark}
    Rich's theorem is now also proven as a corollary because of the $R\hookrightarrow \mathbb Z$ embedding.
\end{remark}

Theorem $4$ parts $(1)$ and $(3)$ are now proved. 
\section{Units of $RC_3$}
We now consider the case when $K$ is a totally real field and $G=C_3=\langle g\rangle$. \\ For any number field $L$ denote $\mu_L$ as the subgroup of roots of unity in $\mathcal{O}_L$. 

Recall, $KC_3=K\eta_0\oplus KC_3\eta_1$. So $KC_{3_t}=KC_3\eta_1\cong K(\omega)$ via $g\mapsto \omega$. We denote $K(\omega)=L$. Then $M_{KC_{3_t}}\cong \mathcal{O}_{L}$. As $RG_t\subseteq  M_{KC_{3_t}} $ the map $g\mapsto\omega$ extends to an injective map from $RG_t$ to $\mathcal{O}_{L}$. We have established the additive structure of $\mathcal{O}_KC_{3_t}$. It correspond with $\mathcal{O}_K[\omega]\subseteq\mathcal{O}_{L}$. We examine its units structure. It must correspond with a subgroup of  $U(\mathcal{O}_{L})$. 

To study the units in $\mathcal{O}_L$, we use the famous Hasse Index, taken from \cite{frohlich1991algebraic}:
\begin{theorem}
       For a totally real field $K$. The Hasse index, $$[U(\mathcal{O}_{L}):U(\mathcal{O}_K) \mu_{L}]=\delta$$ can only be one or two. Moreover, If $\delta=2$ then there exist $\gamma\in\mathcal{O}_{L}$ such that $\gamma^2\in \mathcal{O}_K^\times \mu_{L}.$ ($\gamma\not\in \mathcal{O}_K^\times\mu_L)$. So we can write, $\gamma^2=\epsilon v$, for $\epsilon\in\mathcal{O}_K^\times, v \in \mu_{L}$. If $\epsilon$ is positive under a real embedding of $K$ then $v\not\in\mu_{L}^2$. 
\end{theorem}

We would also rely on common discriminant results, taken from \cite{washington2012introduction}:
\begin{theorem}
    If a natural number, $m$ has at least two distinct prime factors, then $1-\zeta_m$ is a unit. Where  $\zeta_m$ is the $m$-th root of unity.
\end{theorem}
\begin{theorem}
    The absolute discriminant of $\mathbb Q(\zeta_{p^n})$ is $\pm p^k$ where $k=p^{n-1}(pn-n-1)$. 
\end{theorem}
\begin{theorem}
    Given two $\mathcal{O}_K$-lattices $M,N$, the relative discriminants of $K$ over a field extension $L$ have the relations $$disc(N)=disc(M)[M:N]^2.$$ Where the index $[M:N]$ is the determinant ideal of the change of basis matrix between the two lattices. Moreover, if $M\subseteq N$ then $[M:N]\in \mathcal{O}_K$. 
\end{theorem}

We know the trivial units in $RG_t$ will correspond to $U(\mathcal{O}_K) \langle\omega\rangle$ because $g\mapsto \omega$. That means if we want to prove $RG_t$ has only trivial units, we need to prove the following proposition:
\begin{prop}
    Let $u$ be a unit in $\mathcal{O}_{L}$. $u\in \mathcal{O}_K[\omega]$ if and only if $u\in U(\mathcal{O}_K) \langle\omega\rangle$. 
\end{prop}
\begin{proof}
 Clearly, if $u\in U(\mathcal{O}_K) \langle\omega\rangle$ then it is also in $\mathcal{O}_K[\omega]\subseteq \mathcal{O}_{L}$.

For the other direction, we assume to the contrary,
$u\in\mathcal{O}_K[\omega]$ but not in $U(\mathcal{O}_K)\langle\omega\rangle$
That means $u$ can be written as a linear combination, $a+b\omega$ where $a,b\in\mathcal{O}_K$.
 Consider the $\mathcal{O}_K$-lattices $N=\mathcal{O}_K[u]$, and $M=\mathcal{O}_K[\omega]$. Since $u\in \mathcal{O}_K[\omega]$ it must mean $N\subsetneq M$. We compute the discriminant of $L/K$ of each. $disc(M)=3\mathcal{O}_{L}$, while $disc(N)=(\bar{u}-u)^2\mathcal{O}_{L}=(u-\bar{u})^2\mathcal{O}_{L}$ (as ideals in $\mathcal{O}_{L}$). That means, $((u-\bar{u})^2)\mathcal{O}_{L}\subset 3\mathcal{O}_{L}$, and $3|(u-\bar{u})^2$. 
 
Let $\mu_{L}=\langle\zeta\rangle$, the group of roots of unity in $L$. 
Since $K$ is a totally real field, $\mu_K=\langle-1\rangle$. $\mu_{\mathbb Q(\omega)}\leq \mu_{L}$ so $6$ must divide the order of $\mu_{L}$. That is $\zeta$ is a $6m$ root of unity. 

We now consider the case the Hasse index is $1$. 

In that case $u\in U(\mathcal{O}_K)\mu_L$ then $u=\epsilon v$, which we can always multiply by $\epsilon^{-1}$ and since those will be assigned to the trivial units we care only to see what happens to $v$. That means $u\not\in K$ so $u\in\mu_L$ and therefore we can write $v$ as a root of unity $3m$, $u=v=\pm\zeta^j$ where $1\leq j\leq (3m-1).$ That also means $\bar{u}=u^{-1}$ in that case. Therefore, the ideal $((u-\bar{u})^2)\mathcal{O}_{L}$ can be written as $(u-u^{-1})^2)\mathcal{O}_{L}= u^2(1-u^{-2})^2\mathcal{O}_L=(1-u^{-2})^2\mathcal{O}_L$, ($u$ is a unit). Since $u^{-2}\in\mu_L$, it must be of the order of dividing $|\mu_L|=6h$. If $h=1$, then $\mu_L=\langle-\omega\rangle$. In that case $u$ is in $U(\mathcal{O}_K)\langle-\omega\rangle$, which we know corresponds to trivial units, so we can assume $h>1$. Let $|\mu_L|=2^{n_1}\cdot 3^{n_2}\cdot m$ be with $gdc(6,m)=1$. 

If $u^{-2}$ is of order dividing at least two primes, theorem 6 tells us $1-u^{-2}$ is a unit so the ideal $(u-u^{-2})^2\mathcal{O}_L$ is $\mathcal{O}_L$ and can't be contained in $3\mathcal{O}_L$.

If $u^{-2}$ has an order that divides only $p^n$ with $p\neq 3$. Consider the maximal real field of $\mathbb Q(u)$, $\mathbb Q(u+\bar{u})$. Since $[L:K]=2$ $disc(u)_{L/K}=disc(u)_{\mathbb Q(u)/\mathbb Q(u+\bar{u})}$. The $disc(u)_{\mathbb Q(u)/\mathbb Q(u+\bar{u})}$ is a factor of $disc(u)_{\mathbb Q(u)/\mathbb Q}$ as ideals in $\mathcal{O}_L$. That, with theorem 7, tell us that $3$ cannot divide $(u-u^{-2})^2$.

Finally, when $u^{-2}$ has order $3^{n}$, then the order of $u$ ($u$ and $u^{-1}$ have the same order) could be a power of $3$ or $2$ times a power of $3$. 

If it is of power $3$, since it is in $\mu_L$ it must be a $3^{n}$-root of unity with $n>1$. $3\mathbb Z$ is a prime ideal in $\mathbb Z$. However, it is not a prime in $\mathbb Z[\omega]$ (because $3$ ramify in $\mathbb Z[\omega]$ with ramification index $2$. We also know $(1-\omega)\mathbb Z[\omega]$ is a prime ideal in $\mathbb Z[\omega]$ but not a prime in $\mathbb Z[\zeta_9]$, and it has ramification index $3$ because it is a total ramification extension. So, $(1-\zeta_9)\mathbb Z[\zeta_9]$ is a prime ideal in $\mathbb Z[\zeta_9]$ but not in $\mathbb Z[\zeta_{27}]$ and so on. All of that implies, $3\mathbb Z\subsetneqq (1-\zeta_3)\mathbb Z\subsetneqq (1-\zeta_9)\mathbb Z$. Therefore, we can't have $(1-u^{-2})^2\mathcal{O}_L\subset 3\mathcal{O}_L$ unless $u^{-2}$ has order $3$, which we noted it has order $3^{n}>3$. 

If the order of $u^{-2}$ is $2$ times a power of $3$, the ideal $(1-u^{-2})^2\mathcal{O}_L=((1-u^{-1})(1+u^{-1})^2)\mathcal{O}_L=(1+u^{-1})^2\mathcal{O}_L$. (Because $1-u^{-1}$ is a unit from theorem 6. Now, $(-u^{-1})=u^{-2}$, which is of order $3^n$ so the same ramification argument used for $u^{-1}$ when it was of order $3^n$ to have $3\mathcal{O}_L\subsetneqq (1+u^{-1})^2\mathcal{O}_L$. 


The second case is when the Hasse Index is $2$. 

We assume $u\not\in U(\mathcal{O}_K)\mu_L$. According to theorem 5 that means $u^2=\epsilon \zeta$ with $\epsilon\in U(\mathcal{O}_K), \zeta\in\mu_L, \zeta\not\in\mu_L^2$. Also, $(u\bar{u})^2=u^2\bar{u}^2=\epsilon \zeta \bar{\epsilon}\zeta^{-1}=\epsilon\bar{\epsilon}=|\epsilon|^2=\epsilon^2$. ($\epsilon$ is in $K$, real field). Therefore, $u\bar{u}=\epsilon$ and $\bar{\epsilon}=\epsilon$. In this case, $(u-\bar{u})^2=u^2-2u\bar{u}+\bar{u}^2=\epsilon \zeta -2\epsilon+\epsilon\zeta^{-1}=\epsilon(\zeta-2+\zeta^{-1})=\epsilon(\sqrt{\zeta}-\sqrt{\zeta^{-1}})^2=\epsilon\zeta(1-\zeta^{-1})^2=u^2(1-\zeta^{-1})^{2}$. Since $u^2$ is a unit the ideal $u^2(1-\zeta^{-1})^2\mathcal{O}_L=(1-\zeta^{-1})^2\mathcal{O}_L$ are the same. Therefore, $(u-\bar{u})^2\mathcal{O}_L=(1-\zeta^{-1})^2\mathcal{O}_L$ in the $\delta=2$ case. Moreover, since $\zeta\not\in\mu_L^2$, $\zeta^{-1}$ can not have a power of $3$. That is because, all roots of unity with a power of $3$ are a square in $\mu_L$. (In fact, it cannot have an odd order). If the order of $\zeta^{-1}$ is a composite number, $1-\zeta^{-1}$ is a unit, and the ideal $(1-\zeta^{-1})^2\mathcal{O}_L$ is $\mathcal{O}_L$ and therefore cannot be contain on $3\mathcal{O}_L$. Lastly, if the order of $\zeta^{-1}$ is a power of $2$, then $3$ does not divide the discriminant. 

Overall, we can't have $N\subset M$ so when, $u\in U(\mathcal{O}_L)$ and $u\in\mathcal{O}_K[\omega]$ then $u$ must be in $U(\mathcal{O}_K)\langle-\omega\rangle$. 
\end{proof}
Now that our proposition is proved, together with section $3$ and $2$, we can conclude $RG_t$ can only have trivial units. \\ This completes the proof for theorem 4, our main results.  
\\

In the case of $RC_3$ where $R=\mathcal{O}_K$ and $K$ is a totally real field, we see $RC_{3_t}$ has only trivial units. Therefore, $RC_3$ contains only almost trivial units. That is the only case where a group ring $RG$ contains mom trivial units units, but $RG$ contains only trivial. Specifically, we have shown $$U(RC_3)=\{(u\eta_0+v\eta')g | u=v \mod 3; u,v\in U(R); g\in C_3 \}.$$  \\ This is a full description of units in $RC_3$. A similar description of $U(RC_2)$ is given in the beginning of section 3. 

\printbibliography[heading=subbibliography]
\end{document}